\documentclass[12pt]{amsart}
\usepackage{amssymb}
\usepackage{amsmath}
\usepackage{amsfonts}
\usepackage{epsfig}
\usepackage{color}
\usepackage{epstopdf}
\usepackage{graphicx}
\usepackage{amsthm}
\usepackage{enumerate}
\usepackage[mathscr]{eucal}
\usepackage{verbatim}
\usepackage{cite}

 \theoremstyle{plain}
\newtheorem{theorem}{Theorem}[section]
\newtheorem{lemma}[theorem]{Lemma}
\newtheorem{lem}[theorem]{Lemma}
\newtheorem{proposition}[theorem]{Proposition}
\newtheorem{corollary}[theorem]{Corollary}
\newtheorem{defn}[theorem]{Definition}
\newtheorem{rmk}[theorem]{Remark}

\theoremstyle{plain}

\newcommand{\supp}{\text{supp$\!$ }}
\newcommand{\asupp}{\textit{$\mathcal A$supp$\!$ }}

\newcommand{\dist}{\text{dist}}

\newcommand{\be}{\begin{equation}}
\newcommand{\en}{\end{equation}}
\newcommand{\ee}{\end{equation}}
\newcommand{\beal}{\begin{align}}
\newcommand{\eal}{\end{align}}

\title[Fractal Strichartz estimate]
{Fractal Strichartz estimate for  \\ the wave equation}

\author{Chu-Hee Cho}

\author{Seheon Ham}

\author{Sanghyuk Lee}

\address{School of Mathematical Sciences \\
	Seoul National University \\
	Seoul 08826,
 Republic of Korea}
\email{akilus@snu.ac.kr}

\address{School of Mathematics\\
Korea Institute for Advanced Study	 \\
Seoul 02455, Republic of Korea}
\email{hamsh@kias.re.kr}

\address{School of Mathematical Sciences \\
	Seoul National University \\
	Seoul 08826, Republic of Korea}
\email{shklee@snu.ac.kr}

\thanks{This work is supported  by  NRF-2015R1A4A1041675 and NRF- 2015R1A2A2A05000956 (South Korea).}

\begin{document}
\begin{abstract}
We consider Strichartz estimates for the wave equation with
respect to general measures which satisfy certain growth conditions. In  $\mathbb
R^{3+1}$  we obtain the sharp estimate and in higher dimensions improve the previous results.
\end{abstract}
\maketitle

\section{Introduction}

Let us consider the wave equation in $\mathbb R^{n} \times \mathbb R$:
\begin{equation}\label{wveq}\begin{cases}
        \partial_t^2 u -\Delta u=0, \\
      u(x,0)=f,~\,\partial_tu(x,0)=g.
        \end{cases}
\end{equation} 
The space-time estimate for the solution  of \eqref{wveq}  which is called
\emph{Strichartz estimate}  has been proven to be an important tool
in studies of various problems. (See \cite{st, pe, km, ls, MS, kt}.) 
It is well-known that the estimate
\begin{equation}\label{stri}
      \|u\|_{L_t^q(\mathbb R,\, L_x^r(\mathbb R^{n}))}\lesssim\|f\|_{\dot H^s} +
      \|g\|_{\dot H^{s-1}}
      \end{equation}
      holds for $s\ge 0, 2 \le q , r <\infty$ 
      which satisfy
      \[ 
      \frac1q+\frac{n}r=\frac{n}2-s,\,\, \frac1q+\frac{n-1}{2r} \le
      \frac{n-1}4.\]
Here $\dot H^s$ is the homogeneous $L^2$ Sobolev space of order $s$.  See  \cite{FW} for the estimates with $r=\infty$. 
It was Strichartz \cite{st} who first proved the estimate \eqref{stri} when $q=r$. This  was later extended to mixed norm estimates by Pecher \cite{pe}. (Also see \cite{GV}.)
The endpoint cases $(r,q) = (2(n-1)/(n-3),2)$ except $n=3$ were obtained by Keel--Tao \cite{kt}. Klainerman and Machedon \cite{km} showed the failure of \eqref{stri} when $(n,r,q)=(3,\infty,2)$.

\

In this note we consider a generalization of \eqref{stri} 
by replacing the Lebesgue measure with general measure $\mu$. 
More precisely, we  study  the estimate
\begin{eqnarray}\label{fstrichartz}
\|u\|_{L^q(d\mu)}\lesssim\|f\|_{H^s} + \|g\|_{H^{s-1}}.
\end{eqnarray}
Here we denote by $H^s(\mathbb R^{n})$  
the inhomogeneous  $L^2$ Sobolev space of order $s$, which is the space of all tempered distributions $f$ such that $(1 + |\cdot|^2)^{\frac s2} \widehat f \in L^2(\mathbb R^n)$, equipped with the norm 
\[
\|f\|_{H^s(\mathbb R^n) } = \| (1 + |\cdot|^2)^{\frac s2} \widehat f \|_{L^2(\mathbb R^n)}. 
\] 
This kind of estimates  was studied in connection  with problems in geometric measure theory, precisely, the sphere packing problem  (see \cite{Mi, w, ob, ob2}).

Throughout this paper, the measure $\mu$ is assumed to be a nonnegative Borel
regular measure with compact support in $\mathbb R^{n+1}$.
Let us denote  by $\mathfrak M(\mathbb R^{n+1})$  the space of such measures.
In addition we impose uniform growth condition on $\mu$ as follows.

\begin{defn} Let $\alpha\in (0, n+1]$.  For $\mu \in \mathfrak M(\mathbb R^{n+1})$, we say that $\mu$ is $\alpha$-dimensional if there exists a constant $C_\mu$, independent of $x$ and $\rho$, such that
\begin{equation}\label{def-measure} 
\mu(B(x,\rho))\le C_\mu \rho^\alpha \quad\textrm{ for all } x\in \mathbb R^{n+1},\, \rho>0.
\end{equation} 
Here $B(x,\rho)$ denotes the open ball of radius $\rho$ centered at $x$. 
Also we define 
\[
\langle \mu\rangle_\alpha = \sup_{x\in \mathbb R^{n+1},\,\rho>0} \rho^{-\alpha}\mu(B(x,\rho)).
\] 
\end{defn}


For $1\le q\le \infty$
let us set
\begin{equation}\label{necessary}
s(\alpha,q,n)=\begin{cases}
\max \{  \frac{n}2-\frac{\alpha}q ,\,
\frac{n+1}4
\}, \,\, \,\, &\text{ if }{0}<
\alpha \le 1,  \\
\max \{  \frac{n}2-\frac{\alpha}q ,\,
\frac{n+1}4 +\frac{1-\alpha}{2q},\,~~~~~~
\frac{n+2}{4}-\frac{\alpha}{4} \}, \,\, \,\, &\text{ if } {1}<
\alpha \le n,  \\
\max \{  \frac{n}2-\frac{\alpha}q ,\,
\frac{n+1}4+\frac{n+1-2\alpha}{2q},\,\frac{n+1}2-\frac{\alpha}{2}
\},\, &\text{ if } n< \alpha \le n+1.
\end{cases}
\end{equation}
 When $n=2$ 
 Wolff \cite{w} showed that \eqref{fstrichartz}
holds for $\alpha$-dimensional measure $\mu$ if $s>\max(\frac34, 1-\frac{\alpha}4, 1-\frac{\alpha}q)$, $\alpha \in (1,3)$.
Erdo$\breve{\textrm g}$an \cite{er3} improved Wolff's result so that
\eqref{fstrichartz} holds for $s>s(\alpha,q,2)$, $\alpha\in(1,3)$  
and also showed that \eqref{fstrichartz}  generally
fails  if $s<s(\alpha,q,2)$.  
When $n\ge3$, 
Oberlin \cite{ob} obtained \eqref{fstrichartz} for $
\alpha\in (1,n+1)$ provided that $q<\alpha$ and $s>\frac{n-1}2$. 

\

It  is plausible to conjecture that \eqref{fstrichartz} holds if $s> s(\alpha, q,n)$ (see  Proposition \ref{prop:necessary}) but  like other open problems of similar nature  complete resolution seems out of reach at this moment. However, for $n=3$ and $\alpha \in[1,4]$, we obtain the sharp estimate by the following theorem and Proposition \ref{prop:necessary}. 
\begin{theorem}\label{wave} 
Let $n=3$. 
 Also let $\mu$ be  an $\alpha$-dimensional measure. 
Suppose that $u$ is a solution to the equation \eqref{wveq}.
Then \eqref{fstrichartz} holds with 
\begin{equation}\label{scon3}
s> 
\begin{cases}
s(\alpha,q,3), \, \text{ if }\, 2 \le q \le \infty,\\ 
s(\alpha,2,3), \, \text{ if }\, 1 \le q \le 2.
\end{cases}
\end{equation}
Furthermore, the implicit constant in \eqref{fstrichartz} does not depend on particular choice of  $\mu$ as long as $\langle\mu\rangle_\alpha$ is uniformly bounded. 
\end{theorem}

For $\delta >0$ and $R \gg 1$, let us define the truncated cone  $\Gamma_R$   by 
 \[\Gamma_R=\{\xi=(\xi',\xi_{n+1})\in \mathbb R^{n+1}: |\xi'|=\xi_{n+1},
R\le \xi_{n+1}\le 2R \}\] and its $\delta$-neighborhood  by
\[\Gamma_R(\delta)=\{\xi\in \mathbb R^{n+1}: \dist (\xi, \Gamma_R)<\delta  \}.\]
Note that the space-time Fourier transform of $u$ is supported in
the forward and backward light cones. By reflection in frequency spaces, Littlewood-Paley decomposition, and Plancherel theorem,
Theorem \ref{wave} is a consequence of the following. 
(For details, see the end of Section \ref{sec:rescaling}.)

\begin{theorem}\label{thecone} 
 Let $n=3$ and $\mu$ be an $\alpha$-dimensional measure supported  in $\overline{B(0,1)}$.  
Then, for $f$ supported in $\Gamma_R(1)$ there exists a constant $C>0$ such that 
\begin{equation}
\label{frac} \|\widehat f\|_{L^q(d\mu)}\le C  {\langle\mu\rangle_\alpha^{\frac1q }}\, R^s\|f\|_2
\end{equation}
holds if $s$ satisfies \eqref{scon3}.
\end{theorem}

\begin{rmk}
If $n=3$ and  $ \alpha \in [1, 4]$,  \eqref{frac} is  sharp because $s(\alpha,q,3)=s(\alpha,2,3)$ when $1 \le q \le2$. 
If $0 < \alpha <1$, we have $s(\alpha,2,3) > s(\alpha,q,3)$ for $1\le q \le 2$. 
\end{rmk}

In general,  $s(\alpha,q,n)$, $n \ge 3$  provides a lower bound on $s$  for \eqref{frac} to hold.

\begin{proposition}\label{prop:necessary}
Suppose that for any $\alpha$-dimensional measure $\mu$, there is a constant $C>0$ such that \eqref{frac} holds whenever $f$ is supported in $\Gamma_R(1)\subset \mathbb R^{n+1}$. 
Then 
$ s \ge s(\alpha,q, n) $ for $1\le q \le \infty$.
\end{proposition}

 In order to show this, we construct  $\alpha$-dimensional measures and functions for which $\eqref{frac}$ fails  if $s<s(\alpha,q, n) $. For $\alpha\in (1,n+1]$, we modify the examples in \cite{er3} and  for  $\alpha\in (0,1]$ we use the construction in \cite{HL}.  (See Section \ref{sec:sharpness} for details.) For each $\alpha$  some of the conditions appearing in \eqref{necessary}  are natural in view of dilation
and rescaling structure of the estimate \eqref{frac}.  To be precise,  for the case $n <\alpha \le n+1$ in \eqref{necessary}   the condition $s(\alpha,q,n)= \frac{n+1}4 + \frac{n+1-2\alpha}{2q}$ 
 may be understood as a homogeneity condition and     $s(\alpha,q,n)=\frac n2 - \frac \alpha q$ is related to the Knapp type example. 
 A similar statement also applies to the case $1<\alpha\le n$.

\smallskip

In contrast with the Lebesgue measure,  there is no obvious scaling structure for a general $\alpha$-dimensional
measure.  But,  as is to be seen in what follows, if we assume uniform bound which only depends on $\langle\mu\rangle_\alpha$,  the measure $\mu$ can be handled  as if it is homogeneous of degree $\alpha$ 
with respect to isotropic dilation.  This observation plays an important role throughout the paper.  Similar idea was used in \cite{HL} to obtain 
restriction estimate for the curve with respect to general measures other than the Lebesgue measure.

\smallskip

The estimate \eqref{frac} is  closely related to Fourier restriction problem. We make use of the bilinear approach which was extensively utilized  to tackle the restriction problem. Some aspects of our paper are similar to those of \cite{er3}. However,  unlike \cite{er3} we use the induction on scale argument in more direct way without relying on weighted inequality and this enables us to expose  underlying structure more clearly.  Also our argument can be used to recover the results in \cite{er3}.  It is natural to expect that multilinear restriction estimates \cite{BCT} and its recent development (see \cite{BG, G}) can be used for further improvement of the estimate \eqref{fstrichartz}. However, it does not seem likely  that these estimates would give the sharp estimate such as our result in Theorem \ref{wave}.    

\smallskip

In this paper we will prove \eqref{frac} for  $n\ge3$  (see Theorem \ref{thecone1}) while our results are sharp only for $n=3$.
(The sharp estimates for $n=2$ and $1\le\alpha\le3$ was previously obtained by Erdo\u{g}an \cite{er3}.)
When $n \ge 4$,  the necessary condition  $s \ge s(\alpha,q, n)$  (except the equality case)  is sufficient only for  $n< \alpha \le n+1$ or for large $q$ if $\alpha \le n $.

\subsubsection*{Average decay estimate}
It is well-known that \eqref{frac} implies the associated average decay estimate. 
In fact the decay rate is determined by $s(=s(\alpha,2,n))$ of the estimate \eqref{frac}. 
Let $I_\alpha(\mu)$ be an $\alpha$-dimensional energy of $\mu$ which is given  by
$I_\alpha(\mu)=\iint |x-y|^{-\alpha} d\mu(x)d\mu(y)$.
If  $\mu$ is a positive Borel measure supported in the unit ball and satisfies $I_\alpha(\mu)=1$, there exists a constant $C_\alpha > 0$ such that 
\begin{equation}\label{energy}
\int_{\Gamma_1}
|\widehat \mu(R\xi)|^2 d\sigma(\xi) \le C_{\alpha} 
R^{-n + 2 s(\alpha,2, 3)+\varepsilon} 
\end{equation}
for $R>1$ and any $\varepsilon >0$. 
This can be shown by the argument in \cite{w2} which makes use of Lemma 1.5 and duality.
(See Section 2 in \cite{w2} for details.) 

The  average decay estimate over the sphere has been studied in
connection with Falconer's distance set conjecture. (See \cite{Matt1}, \cite{b1},
\cite{w2}, \cite{er2}, \cite{er1}, \cite{LR} and references therein.)

\subsubsection*{The sphere packing problem}
Let $S(x,r)$ be a sphere in $\mathbb R^{n}$ with center $x$ and radius $r > 0$. We denote the Hausdorff dimension by $dim_{H}$ and the $d$-dimensional Lebesgue measure by $|\cdot|_{d}$. Theorem \ref{thecone} immediately implies the following.  
\begin{corollary}\label{corollary}
Let $E \subset \mathbb R^{n}$ and   $P
$ be a Borel set in $\mathbb R^{n+1}$ with $dim_H(P) > 1$. 
Assume that $|S(x,r) \cap E|_{n-1} > 0$ for any $(x,r) \in P$. 
Then $|E|_n >0$.
\end{corollary}

Wolff \cite{w} proved that {Corollary \ref{corollary}} is valid when $n=2$. When $n \ge 3$, Oberlin \cite{ob} showed that the statement holds by obtaining estimate for the spherical average.


\subsubsection*{Organization of the paper}
In Section \ref{sec:bilinear} we prove a bilinear version of \eqref{frac}, 
which is obtained by an adaptation of the induction on scale argument. In Section  \ref{sec:rescaling}
   proofs of Theorem \ref{thecone}  and  Theorem \ref{wave} are given. 
In Section \ref{sec:sharpness}, we discuss the necessary conditions in Proposition \ref{prop:necessary}.

 %

%

\section{Bilinear estimates}\label{sec:bilinear}
In this section we prove a bilinear version of the estimate \eqref{frac}  which is closely related to bilinear restriction estimate for the cone (\cite{tvv, tv1, w3, tao-bilinear, le}). 
Under an additional transversality condition, bilinear estimate gives formally better estimate than linear one by weakening  Kakeya compression phenomena. 


\begin{defn}  For a function $f$ which is supported away from the origin  we define the angular support $\asupp f$  by 
\[ \asupp f=\Big\{  \frac \xi{|\xi|}  :   \xi\in \supp f  \Big \}. \] 
\end{defn}

The following may be regarded as a generalization of bilinear restriction estimate for the cone in \cite{w3}.

\begin{theorem}\label{bifrac}
Let $R\gg 1$ and let $\mu$ be an $\alpha$-dimensional
measure supported in $\overline{B(0,1)}$. Suppose that  $f$ and $g$
are supported  in $\Gamma_R(1)$ and
\begin{equation}\label{angular-separation}
\dist( \asupp f, \asupp g)\ge \frac1{100}. 
\end{equation}
For $2\le
q\le\infty$, there is a constant {$C = C(\beta,n)$} 
such that
\begin{equation}\label{bi}
\left(\int |\widehat f\,\, \widehat g|^{\frac q2} d\mu\right)^\frac2q\le
CR^{2\beta} {\langle\mu\rangle_\alpha^{\frac2q}}\|f\|_2 \|g\|_2
\end{equation}
for any $\beta > \beta(\alpha,q):= \max\{ \frac{n}2-\frac{\alpha} q, \,\frac{3n+1-2\alpha}{8} \} $. 
\end{theorem}

Since $|\widehat f|\lesssim  R^{\frac n2} \|f\|_2$ and $|\widehat g|\lesssim R^{\frac n2} \|g\|_2$,   \eqref{bi} trivially holds with  $2\beta\ge n$. 
It is easy to verify that the condition $\beta \ge \frac{n}2-\frac{\alpha} q$ is necessary by adopting the examples which are used to show 
the necessity of the condition \eqref{scaling}. 
Another necessary condition for \eqref{bi} with $\alpha >2$ is 
\begin{equation}\label{bilinear-condition}
\beta\ge \frac{n-1}4-\frac{\alpha-2}{2q}. 
\end{equation}
To see this, we consider the squashed
cap function (see \cite{tvv}) and a suitable measure. Precisely, 
let us set $d\mu=\psi(x)|x''|^{-n-1+\alpha}
dx_1dx_2dx''$, $x=(x_1, x_2, x'')$  for a smooth cutoff function $\psi$. 
Then $\mu$ is an $\alpha$-dimensional measure if $n+1\ge
\alpha>2$. 
Considering a pair of characteristic functions $f, g$ supported in $\Gamma_R(1)$ with large
angular separation and dimensions  $1\times 1\times
\overset{\,\,\,\,\,(n-1)\text{ times}}{\sqrt R \times \cdots \times \sqrt R}$ it is easy to show \eqref{bilinear-condition}.

\subsubsection*{Localization  to a smaller cube}  Improvement due to localization is important for the induction on scale argument. 
In the following lemma  we make it precise how localization to smaller cubes affects  the estimate \eqref{bi}.

\begin{lem}\label{localization}  Let  $\beta\ge 0$, $q\ge2$. 
Suppose that \eqref{bi} is valid for any $\alpha$-dimensional measure $\mu$ supported in $\overline{B(0,1)}$, and for  any  $f$,
$g$ supported in $\Gamma_R(A)$ for some $A \sim 1$ satisfying the condition \eqref{angular-separation}. Then, for $x_0 \in \mathbb R^{n+1}$ and $\rho\le (10A)^{-1}$,
\begin{align}\label{local1}
\left(\int_{B(x_0,\rho)} |\widehat f\,\, \widehat g|^{\frac q2} d\mu\right)^\frac2q\le
 C {\langle\mu\rangle_\alpha}
 ^{\frac2q}R^{2\beta}\rho^{\frac{2\alpha} q+2\beta-n} \|f\|_2 \|g\|_2.
 \end{align}
\end{lem}

\begin{proof}  
By translation we may
assume $x_0=0$ without loss of generality. 

Since the map $\phi \mapsto \rho^{-\alpha} \int \phi(\rho^{-1} x) d\mu(x)$   is a positive linear functional on $C_c(\mathbb R^{n+1})$, by the Riesz representation theorem there exists a Radon measure $\mu_\rho$ such that
\[\int \phi(y) d\mu_\rho(y)= \rho^{-\alpha}\int \phi(\rho^{-1} x) \chi_{B(0,\rho)} (x)d\mu(x)
\] 
for any continuous function $\phi$. 
Then it is obvious that
$\mu_\rho$ is supported in $ \overline{B(0,1)}$ and $\langle\mu_\rho\rangle_\alpha \le
\langle\mu\rangle_\alpha$.
Let us set  $f_\rho=\rho^{-n-1} f(\cdot/\rho)$, $g_\rho=\rho^{-n-1} g(\cdot/\rho)$. Then,  
it follows  that
\begin{align}\label{local2}
\int_{B(0,\rho)} |\widehat f\, \widehat g |^{\frac q2} d
\mu 
=\rho^\alpha \int_{{B(0,1)}} |\widehat f_\rho\,\widehat g_\rho|^{\frac q2}
d \mu_\rho .
\end{align}
We note that $f_\rho$ and $ g_\rho$ are supported in the set $\Gamma_{R\rho}(A \rho).$ Since the measure $ \mu_\rho$ is supported in $\overline{B(0,1)}$,
we may put a harmless smooth function $\eta$ in the integral which
satisfies $\widehat \eta\sim 1$ on $\overline{B(0,1)}$ and $\supp \eta \subset
B(0,\frac12)$. 
Then we see
\begin{align}  \label{local3}
\begin{aligned}
\int |\widehat f_\rho(x)\widehat g_\rho(x)|^{\frac q2} d \mu_\rho
(x)&\sim \int |\widehat\eta ^2 \widehat f_\rho(x)\widehat
g_\rho(x)|^{\frac q2} d \mu_\rho (x)
\\&=\int | \widehat {\eta\ast
f_\rho}(x)\widehat{ \eta\ast g_\rho}(x)|^{\frac q2} d \mu_\rho (x).
\end{aligned} \end{align}
Note 
that $\eta\ast f_\rho$ and $\eta\ast g_\rho$ are 
supported in $\Gamma_{R\rho}(1)$. 
Hence, from \eqref{bi} we have 
\[
\Big(\int | \widehat {\eta\ast
f_\rho}(x)\widehat{ \eta\ast g_\rho}(x)|^{\frac q2} d \mu_\rho
(x)\Big)^\frac2q \le C \langle \mu\rangle_{\alpha}^{\frac2q}(R\rho)^{2\beta} \|\eta\ast f_\rho\|_2\|\eta\ast
g_\rho\|_2.
\]
%
Since $f_\rho$ and $g_\rho$ are supported in
$\Gamma_{R\rho }( A\rho)$, 
 $\|\eta\ast f_\rho\|_2\le
C\rho^{-n/2}\|f\|_2$\footnote{This follows from H\"older inequality and easy estimates $ \|\eta\ast f_\rho\|_1\lesssim
\|f\|_1$,  $\|\eta\ast f_\rho\|_\infty \lesssim \rho^{-n}
\|f\|_\infty$.  }   and $\|\eta\ast g_\rho\|_2\le C\rho^{-n/2}
\|g\|_2$. 
Therefore \eqref{local1} follows by \eqref{local2} and \eqref{local3}.
\end{proof}

\subsubsection*{Decomposition of $f$ and $g$}  For the  proof of Theorem \ref{bifrac}  we adapt the induction on scale argument which has been used to prove sharp bilinear restriction estimate \cite{w3, tao-bilinear, le}.  The following is actually a rescaled version of  \cite[Lemma 3.5] {w3}  (also see  \cite{ tao-bilinear, le} for a simpler proof based on wave packet decomposition, especially the proof of Theorem 1.3 in \cite{le}). 

\begin{lemma}\label{wavepacket}
Let $R \gg 1$ and $0<\delta\le 1$. Let $\{\mathbf
q\} $ be a collection of cubes of sidelength $\sim R^{-\delta}$ which partition the cube
$[-1,1]^{n+1}$.  Suppose $f$ and $g$ are supported in $\Gamma_R(1)$ and satisfy \eqref{angular-separation}. Then for each $\mathbf
q$, $f$ and $g$ can be decomposed such that
\[ f=f_{\mathbf q}+ f_{\overset{\not\sim}{\mathbf q}} \quad \text{and} \quad  g=g_{\mathbf q}+ g_{\overset{\not\sim}{\mathbf q}} \]
with $f_{\mathbf q}$,
$f_{\overset{\not\sim}{\mathbf q}}$, $g_{\mathbf q}$,
$g_{\overset{\not\sim}{\mathbf q}} $ supported in $\Gamma_R(C_d)$ for some $C_d\sim 1$
and,  for $0<\epsilon\ll \delta$, 
\begin{align}
\label{l2sumq}
 &\sum_{\mathbf q} \|f_{\mathbf q}\|_2^2\le CR^\epsilon
\|f\|_2^2,\qquad \sum_{\mathbf q} \|g_{\mathbf q}\|_2^2\le
CR^\epsilon \|g\|_2^2,
\\
\label{bil2} \|\widehat {f_{\mathbf
q}}\widehat{g_{\overset{\not\sim}{\mathbf q}}}\|&_{L^2(\mathbf q)},
\quad \|\widehat{f_{\overset{\not\sim}{\mathbf q}}}\widehat
{g_{\mathbf q}}\|_{L^2(\mathbf q)}, \quad
\|\widehat{f_{\overset{\not\sim}{\mathbf
q}}}\widehat{g_{\overset{\not\sim}{\mathbf q}}}\|_{L^2(\mathbf q)}
\le C R^{\frac{n-1}4+c\delta+\epsilon} \|f\|_2\|g\|_2
\end{align}
with $c$ independent of $\delta$ and $\epsilon$.
\end{lemma}

\begin{proof}   Let $F$ and $G$ be functions supported in $\Gamma_1(R^{-1})$. 
Then, by using  wave packet decomposition\footnote{It should be noted that we are working with functions instead of extension  operators.  Nevertheless, the argument in \cite{le} works without extra difficulty.} (\cite{tao-bilinear, le}), on $B(0, R)$   we  may write \[\widehat F = \sum_{w_1}c_{w_1}p_{w_1} \quad \text{and} \quad \widehat G = \sum_{w_2}c_{w_2}p_{w_2},\]
where $c_{w_i}$ and $p_{w_i}$ satisfy the following:  
\begin{align*}
&(1)\,\, \big(\sum_{w_1 \in W_1}|c_{w_1}|^2\big)^{1/2} \le C\|F\|_2,\,\,\big(\sum_{w_2 \in W_2}|c_{w_2}|^2\big)^{1/2} \le C \|G\|_2; \\ 
&(2)\,\,\, \text{If}\,\,  w_i\in W_i, \, \supp \widehat {p_{w_i}} \subset \{\xi : \xi =  v_i + O(R^{-1/2})\};\\ 
&(3)\,\,\, \text{If} \,\,  \dist (x,T_{w_i}) \ge R^{1/2 + \delta}, \,\, |p_{w_i}(x)|\le CR^{-100n}; \\ &(4) \,\,\,
\text{For any subset}\,\, S \subset W_{i},\, \|\sum_{w_i \in S}p_{w_i}\|_2^2 \le C \# S.
\end{align*}
Here $W_i$ is the set of all pairs $w_i = (y_i,v_i)$ with $y_i \in R^{1/2}\mathbb Z^n$ and $v_i \in R^{-1/2}\mathbb Z^n \cap \{  v \in \mathbb R^n: 2^{-1} \le  |v| \le 2^2\}  )$ and  $T_{w_i}$ is defined by 
\[T_{w_i} =\big \{(x',x_{n+1}) \in \mathbb R^n \times \mathbb R : |x_{n+1}| \le R, \big|x'-\big(y_i+x_{n+1}\frac{v_i}{|v_i|}\big)\big| \le R^{1/2}\big\}\] for each $w_i \in W_i$, $i=1,2$. 

Partition $[-R,R]^{n+1}$ into  $R^{1-\delta}$-cubes $\{R\mathbf q\}$ which are essentially disjoint. Then,  \[\|\widehat F\,\widehat{G}\|_{L^q(B(0,R))} \le \sum_{\mathbf q}\|\widehat F\,\widehat{G}\|_{L^q( R\mathbf q)}.\]
Now we use the relation $\sim$ between $w_i$ and  $R^{1-\delta}$ cube $R\mathbf q$ which was introduced in \cite{tao-bilinear} (also see \cite{le}).  For $\mathbf q$ fixed,  we say  $w_i\sim \mathbf q$ if $w_i\sim R\mathbf q$ and, otherwise, we say $w_i \not\sim \mathbf q$. 
We keep the same notation $\sim$ since it does not cause any ambiguity.
We now set \[\widehat{F_{\mathbf q}} = \sum_{w_1 \sim \mathbf q} c_{w_1} p_{w_1},\quad \widehat{F_{\overset{\not\sim}{\mathbf q}}} = \sum_{w_1\not\sim \mathbf q} c_{w_1} p_{w_1}, \quad
\widehat{G_{\mathbf q}} = \sum_{w_2 \sim \mathbf q} c_{w_2} p_{w_2}\quad \widehat{G_{\overset{\not\sim}{\mathbf q}}} 
= \sum_{w \not\sim \mathbf q} c_{w_2} p_{w_2}.\]	
Then, by repeating the argument in \cite{tao-bilinear, le} it is not difficult to see that
\begin{align*}
&\sum_{\mathbf q} \|\widehat {F_{\mathbf q}}\|_2^2\lesssim R^\epsilon
\|F\|_2^2, \quad \sum_{\mathbf q} \|\widehat {G_{\mathbf q}}\|_2^2\lesssim
R^\epsilon \|G\|_2^2;
\\
\|\widehat {F_{\mathbf
		q}}\widehat{G_{\overset{\not\sim}{\mathbf q}}}\|_{L^2(R\mathbf q)},
		&
\quad \|\widehat{F_{\overset{\not\sim}{\mathbf q}}}\widehat
{G_{\mathbf q}}\|_{L^2( R\mathbf q)}, \quad
\|\widehat{F_{\overset{\not\sim}{\mathbf
			q}}}\widehat{G_{\overset{\not\sim}{\mathbf q}}}\|_{L^2( R\mathbf q)}
\lesssim R^{-\frac{n+3}4+c\delta+\epsilon} \|F\|_2\|G\|_2.
\end{align*} 
Since we are dealing with functions $F$ and $G$ which are supported in $\Gamma_1(R^{-1})$ instead of being supported on the surfaces,
we obtain an extra power $R^{-1}$ compared to the extension operator (\textit{cf.} \cite{tao-bilinear, le}).

Since $f, g$ are supported in $\Gamma_R(1)$ and satisfy \eqref{angular-separation},  we apply the above to   \[F = R^{\frac{n+1}{2}}f(R\cdot) \quad \text{and}\quad  G = R^{ \frac{n+1}{2}}g(R\cdot). \]
Then,  the change of variable $x\to   x/R $  gives the desired estimates \eqref{l2sumq} and \eqref{bil2}.	
\end{proof}

To prove Theorem \ref{bifrac} we make use of the following simple lemma.
\begin{lem}\label{fractal} Let $R\gg 1$ and $\mu$ be an  $\alpha$-dimensional
measure supported in $\overline{B(0,1)}$.
Set $\phi_R=R^{n+1}\phi(R\cdot)$ for a Schwartz function $\phi$. Then, for any $r \ge 1$ and any Schwartz function $\phi$,
 there is a constant $C$, independent of $\mu$, such that
\begin{equation}\label{convolution}
\|\phi_R\ast d\mu\|_{r}\le C\langle\mu\rangle_\alpha R^{(n+1-\alpha)(1-\frac1r)}.
\end{equation}
\end{lem}

\begin{proof} By rapid decay of $\phi$ it follows that 
$\phi_{R}\le C_N R^{n+1}\sum_{j\ge 0} 2^{-Nj}\chi_{B(0,2^{j}R^{-1})}$ for all $N >0$.
{Hence, recalling the definition of $\langle\mu\rangle_\alpha$ in Definition 1.1} and choosing a sufficiently large $N$,  we have
\begin{align*}
\phi_R\ast d\mu(x)&\le C_N R^{n+1}\sum_{j\ge 0}
2^{-Nj}\mu(B(x,2^{j}R^{-1}))\\
&\le C_N \langle\mu\rangle_\alpha R^{n+1-\alpha}\sum_{j\ge 0} 2^{(\alpha-N)j}
\le C_N \langle\mu\rangle_\alpha R^{n+1-\alpha}.
\end{align*}
This yields \eqref{convolution} with $r=\infty$. 
Since $\mu$ is supported in $\overline{B(0,1)}$, we also have $\|\mu\|\le 2^\alpha\langle\mu\rangle_\alpha$  where $ \|\mu\| = \sup \{ |\int f(x) d\mu(x) | : |f|\le 1,\,  f \in C(\mathbb R^{n+1}) \}$.  
So, from Young inequality  $\|\phi_R\ast d\mu\|_1\le \|\phi_R\|_1 \|\mu\|\lesssim \langle\mu\rangle_\alpha$.
Hence, \eqref{convolution} follows from $\| \phi_R \ast d\mu\|_r^r \le \|\phi_R \ast d\mu \|_\infty^{r-1} \|\phi_R \ast d\mu\|_1$ for any $r\ge 1$.
\end{proof}

\begin{proof}[Proof of Theorem \ref{bifrac}]

It is sufficient to show the case $2 \le q \le 4$. 
Extension to $q\ge 4$ can be obtained by interpolation with
the {trivial estimate}
\[
\| \widehat f\,\,\widehat g\|_{L^\infty {(d\mu)}} \lesssim R^{n} \| f \|_{L^2} \| g \|_{L^2},
\]
which is valid since $\widehat f$ and $\widehat g$ are continuous.

In what follows we prove the implication from \eqref{bi}  to  \eqref{imply}. 
We assume that \eqref{bi} holds for some $\beta>0$ and $f,g$ supported in $\Gamma_R(1)$. As is mentioned before, this is true with a large $\beta>0$. Since $\mu$ is supported in $\overline{B(0,1)}$,  we have 
\[\Big(\int |\widehat f\,\,\widehat g|^{q/2}
d\mu  \Big)^\frac2q
\le \sum_{\mathbf q} \Big( \int_{\mathbf q} | \widehat
f\,\,\widehat g|^{q/2} d\mu \Big)^\frac2q.\] 
 Using  the decomposition in Lemma \ref{wavepacket}, we have
\begin{equation} \label{qube}
\begin{aligned}
\Big(\int |\widehat f\,\,\widehat g|^{q/2}
d\mu  \Big)^\frac2q
&
\le \sum_{\mathbf q} \Bigg[\Big( \int_{\mathbf q} | \widehat
f_{\mathbf q}\widehat g_{\mathbf q}|^{q/2} d\mu \Big)^\frac2q +
\Big( \int_{\mathbf q} | \widehat {f_{\mathbf
q}}\widehat{g_{\overset{\not\sim}{\mathbf q}}}|^{q/2}
d\mu\Big)^\frac2q
\\ 
&\qquad +
 \Big(
\int_{\mathbf q} | \widehat{f_{\overset{\not\sim}{\mathbf
q}}}\widehat {g_{\mathbf q}}|^{q/2} d\mu \Big)^\frac2q+
 \Big(
\int_{\mathbf q} | \widehat{f_{\overset{\not\sim}{\mathbf
q}}}\widehat{g_{\overset{\not\sim}{\mathbf q}}}|^{q/2} d\mu\Big)^\frac2q\Bigg].
\end{aligned}\end{equation}  
By Lemma \ref{localization}, it follows that 
\[\Big( \int_{\mathbf q} | \widehat
f_{\mathbf q}\widehat g_{\mathbf q}|^{q/2} d\mu \Big)^\frac2q \le
 C \langle\mu\rangle_{\alpha}^{\frac 2q} R^{2\beta}R^{\delta(n-\frac{2\alpha}q-2\beta)}\|f_{\mathbf q}\|_2
\|g_{\mathbf q}\|_2.\]
Schwarz inequality and \eqref{l2sumq}
give
\begin{equation}\label{local}
\sum_{\mathbf q} \Big( \int_{\mathbf q} | \widehat
f_{\mathbf q}\widehat g_{\mathbf q}|^{q/2} d\mu \Big)^\frac2q \lesssim   C \langle\mu\rangle_{\alpha}^{\frac 2q} R^{2\beta+\epsilon}R^{\delta(n-\frac{2\alpha}q-2\beta)}\|f\|_2
\|g\|_2.
\end{equation}

For the other terms in \eqref{qube}
we use \eqref{bil2}.  Let $\eta$ be a Schwartz function such that $\widehat \eta(\xi)=1$ if
$|\xi|\le 10C$ and $\supp\widehat \eta\subset B(0,20C)$ and set  $\eta_R=R^{n+1} \eta(R\cdot)$.  Since $f_{\mathbf
q}\ast g_{\overset{\not\sim}{\mathbf q}}$ is supported in a ball of
radius $10R$,   
we have $ |\widehat {f_{\mathbf
q}}\widehat{g_{\overset{\not\sim}{\mathbf q}}} | =  |(\widehat
{f_{\mathbf q}}\widehat{g_{\overset{\not\sim}{\mathbf q}}}  )\ast \eta_R | 
\lesssim  ( |\widehat{f_{\mathbf q}}\widehat{g_{\overset{\not\sim}{\mathbf q}}}  |^{q/2} \ast \eta_R )^{2/q}$. 
By \eqref{convolution}, we thus obtain
 \begin{align*}
\Big( \int_{\mathbf q} | \eta_R \ast \big(\widehat {f_{\mathbf
q}}\widehat{g_{\overset{\not\sim}{\mathbf q}}}\big)|^{q/2}
d\mu \Big)^\frac2q & \le C\Big( \int_{\mathbf q} | \big(\widehat
{f_{\mathbf q}}\widehat{g_{\overset{\not\sim}{\mathbf
q}}}\big)|^{q/2} |\eta_R| \ast d\mu \Big)^\frac2q
\\
&\le C
\begin{cases}
 \|\widehat {f_{\mathbf
q}}\widehat{g_{\overset{\not\sim}{\mathbf q}}}\|_{L^2(\mathbf q)}
\|  |\eta_R| \ast d\mu\|_\infty^\frac12& \text { if } q=4\\
 \|\widehat {f_{\mathbf
q}}\widehat{g_{\overset{\not\sim}{\mathbf q}}}\|_{L^2(\mathbf q)} \|
|\eta_R| \ast d\mu\|_2& \text { if } q=2.
\end{cases}
\end{align*}
By \eqref{bil2} it follows that  for $2\le q\le 4$, 
\begin{align*} 
\Big( \int_{\mathbf q}  \big| \widehat {f_{\mathbf
q}}\widehat{g_{\overset{\not\sim}{\mathbf q}}}\big|^{q/2}
d\mu \Big)^\frac2q &\le C\|\widehat {f_{\mathbf
q}}\widehat{g_{\overset{\not\sim}{\mathbf q}}}\|_{L^2(\mathbf q)} \langle\mu\rangle_\alpha^{\frac2q}
R^{\frac{n+1-\alpha}2}
\\
&\le C \langle\mu\rangle_\alpha^{\frac2q} R^{\frac{3n+1-2\alpha}4+c\delta+\epsilon} \|f\|_2\|g\|_2.
\end{align*}
 The  terms 
$\widehat{f_{\overset{\not\sim}{\mathbf q}}}\widehat {g_{\mathbf
q}}$ and $\widehat{f_{\overset{\not\sim}{\mathbf
q}}}\widehat{g_{\overset{\not\sim}{\mathbf q}}}$ 
can be treated similarly. Since 
$ \#\{\mathbf q\}\le CR^{(n+1)\delta}$, 
combining the estimates for these terms, we get
\begin{align*}
\sum_{\mathbf q}  \Big( \int_{\mathbf q} | \widehat {f_{\mathbf
q}}\widehat{g_{\overset{\not\sim}{\mathbf q}}}|^{q/2}
d\mu\Big)^\frac2q &+
\Big(\int_{\mathbf q} | \widehat{f_{\overset{\not\sim}{\mathbf
q}}}\widehat {g_{\mathbf q}}|^{q/2} d\mu \Big)^\frac2q+
 \Big(
\int_{\mathbf q} | \widehat{f_{\overset{\not\sim}{\mathbf
q}}}\widehat{g_{\overset{\not\sim}{\mathbf q}}}|^{q/2} d\mu\Big)^\frac2q
 \\ &\le 
  C \langle\mu\rangle_\alpha^{\frac2q} R^{\frac{3n+1-2\alpha}4+{\tilde c \delta}+\epsilon}
\|f\|_2\|g\|_2.
\end{align*}
 Here $\tilde c = c +(n+1)$.  
By this and \eqref{local} it follows from \eqref{qube} that for any $\epsilon>0$, 
 \begin{align}\label{imply} 
\Big(\int |\widehat f\,\,\widehat g|^{q/2} d\mu \Big)^\frac2q
&\le C  \langle\mu\rangle_\alpha^{\frac2q} R^{\max \left\{ 2\beta + \delta\left( n-\frac{2\alpha}q-2\beta \right),\, \frac{3n+1-2\alpha}4 + \tilde c \delta \right\}+\epsilon}  \|f\|_2\|g\|_2. 
\end{align}

Hence we have shown  that $\eqref{bi}$ implies \eqref{imply}.   
If we have \eqref{bi} for $\beta= \beta_i$, then we see that \eqref{bi} holds with
\[ 
\beta  = \frac{3n +1 -2\alpha}{8} + \tilde c \left( \frac{8\beta_i - 3n-1+2\alpha}{8(\tilde c - n +2\alpha/ q +2\beta_i)} \right) := \beta_{i+1}
\]
by choosing $\delta = \frac{8\beta - 3n-1+2\alpha}{4(\tilde c - n +2\alpha/ q +2\beta)}$.
Iterating this implication we obtain a sequence $\{\beta_i\}_{i=0}^\infty$.  
Note that the sequence  $\{\beta_i\}_{i=0}^\infty$ is strictly decreasing  as long as 
$- n +2\alpha/ q +2\beta_i > 0$ i.e. $\beta_i > n/2 - \alpha/q$.
Since $\beta_i$ converges to $(3n+1 -2\alpha)/8$,
we conclude that \eqref{bi} holds for $2\le q\le 4$ and for any $\beta > \max\{ \frac{n}2-\frac{\alpha} q, \,\frac{3n+1-2\alpha}{8} \}$.
This completes the proof.
\end{proof}

\section{Proof of Theorem \ref{thecone}: Rescaling}\label{sec:rescaling}

In this section we prove the following theorem by deducing  linear estimate from the bilinear one \eqref{bi}.

\begin{theorem}\label{thecone1} 
 Let $n\ge3$ and $\mu$ be an $\alpha$-dimensional measure supported  in $\overline{B(0,1)}$.
Then, for $f$ supported in $\Gamma_R(1)$ there exists a constant $C>0$ such that \eqref{frac} holds for $s$ satisfying
\begin{align*}
s  > \widetilde s (\alpha,q,d ) :=
\begin{cases}
\,    \max \{\frac {n}2-\frac\alpha q, \frac{n+1}4,\frac{3n+1}8-\frac\alpha 4\}, &\text{if } 0<\alpha\le1,\\
\,          \max\{\frac {n}2-\frac\alpha q, \frac{n+1}4+\frac{1-\alpha}{2q},\frac{3n+1}8-\frac\alpha 4\}, &\text{if } 1 <\alpha\le n,\\
\, \max \{ \frac {n}2-\frac\alpha q, \frac{n+1}4+\frac{n+1-2\alpha}{2q},\frac{n+1-\alpha}{2} \},
             & \text{if } n <\alpha \le n+1.
           \end{cases}
\end{align*}
\end{theorem}

For $n=3$, this immediately implies Theorem \ref{thecone}.
\begin{proof}[Proof of Theorem \ref{thecone}]
It suffices to check that $\widetilde s (\alpha,q,3) =  s(\alpha,q,3)$ for $ 0 <\alpha \le 3$.  
If $0<\alpha\le 1$,  it is easy to check 
$\widetilde s (\alpha, q, 3) = n/2 - \alpha/q = s(\alpha,q,3)$ for $q \ge 2$.
If $1<\alpha\le 3$, we see that $\widetilde s(\alpha,q,3) = s(\alpha,q,3)$ because $(3n+1-2\alpha)/8 = (n+2-\alpha)/4$ when $n=3$.
This completes the proof.
\end{proof}

To prove Theorem \ref{thecone1}, we begin with a Whitney type decomposition to exploit bilinear estimate. (See \cite{tvv,w3}.)

\subsubsection*{Whitney type decomposition of $\mathbb S^{n-1}\times \mathbb S^{n-1}$}
For each $j\ge 1$, let us dyadically divide the sphere
$\mathbb S^{n-1}$ into $O(2^{(n-1)j})$  
caps ${\theta^j_k}$ of diameter $\sim 2^{-j}$.
We will write $\theta^j_k\approx \theta^j_{k^\prime}$ to mean that
$\theta^j_k$ and $\theta^j_{k^\prime}$ are not adjacent but have
adjacent parent caps of diameter $\sim$ $2^{1-j}$. 
Then, by Whitney  decomposition of $\mathbb S^{n-1}\times \mathbb S^{n-1}$  away from its diagonal $D = \{ (x,x) : x \in \mathbb S^{n-1} \}$
we have 
\[
\mathbb S^{n-1}\times \mathbb S^{n-1}\setminus D
=\bigcup{}_{ j\ge 1 } \bigcup{}_{(k,k^\prime):\theta^j_k\approx
\theta^j_{k^\prime}}\theta^j_k\times \theta^j_{k^\prime}.
\] 
Let $D(R^{-1/2})$ be a $O(R^{-1/2})$ neighborhood of $D$.
Then $\mathbb S^{n-1}\times \mathbb S^{n-1}\setminus D(R^{-1/2})$ can be covered by $\theta^j_k\times \theta^j_{k^\prime}$ such that $\theta^j_k\approx \theta^j_{k^\prime}$ and $1\le j \le \log R^{1/2}$.
Also $D(R^{-1/2})$ is covered by a union of disjoint cubes $\theta^{j_\circ}_k\times \theta^{j_\circ}_{k}$ of sidelength $O(R^{-1/2})$. So, $2^{-j_\circ}\sim 1/{\sqrt R}$.

\

For $f$ supported away  from the origin we set
\[f_k^j(\xi,\tau)=\chi_{\theta_k^j}({\xi}/{\tau}) f(\xi,\tau), \quad (\xi,\tau)\in \mathbb R^n\times\mathbb R.\]
Then by the above decomposition we  have
\begin{equation}\label{decomp}
| \widehat f |^2 \le \sum_{j=1}^{\log R^{1/2}}\sum_{(k,k^\prime):\theta^j_k\approx \theta^j_{k^\prime}}
|\widehat {f_k^j}\widehat {f_{k'}^j}| + \sum_{k} |\widehat{f_k^{j_\circ}}|^2.
\end{equation}

Note that  the diameters of $\asupp f_k^j$ and $\asupp f_{k'}^j$ are $O(2^{-j})$. In order to handle the first sum we need the following.


\begin{lem}\label{rescaling}  
Let $1\le j \le \log(R^{1/2})$.
Suppose $f$ and $g$ are supported in $\Gamma_R(1)$  and the diameters of  $\asupp f$ and $\asupp g\le 2^{-j-2}$. If 
$
\dist (\asupp  f,  \asupp  g) \ge 2^{-j},
$
there exists a constant $C>0$ such that
\begin{equation}\label{sep}
\| \widehat {f}\widehat {g}\|_{L^{\frac q2}(d\mu)}\le
C \langle \mu\rangle_\alpha^{\frac 2 q} \, R^{2\beta}\, 2^{\gamma j} \,\|f\|_2 \|g\|_2,
\end{equation} 
where
$\beta > \beta(\alpha,q)$ and 
\[\gamma=\gamma(\alpha,\beta, n)=\begin{cases}
-4\beta+ 2(n+1-2\alpha)/q+n+1, &\text{ if } n< \alpha \le n+1,\\
-4\beta + 2(1-\alpha)/q +n+1,  &\text{ if } 1\le \alpha \le n,\\
-4\beta+n+1,  &\text{ if } 0<\alpha\le 1.
\end{cases}\]
\end{lem}

\begin{proof}
By rotation we may assume that
 $\asupp f$ and $\asupp g$ are contained in 
$O(2^{-j})$ neighborhood of $ (0,\dots, 0,\frac 1{\sqrt 2}, \frac1{\sqrt 2}) \in \mathbb R^{n+1}$. 
More precisely, we
use the null coordinates
\[\xi''=(\xi_1,\xi_2, \dots, \xi_{n-1}), \, \sigma=\frac{\xi_{n+1}-\xi_n}{\sqrt 2},\, \tau=\frac{\xi_{n+1}+\xi_n}{\sqrt 2}.\]
Then the cone $\Gamma_R$ can be written as $|\xi''|^2=2\sigma\tau$. So
\[\Gamma_{R}(1)\subset\{(\xi'',\sigma,\tau):  |\xi''|^2=2\sigma\tau+ O(1),\, \tau\sim R \}.\]
Additionally, after a slight adjustment we may assume that $\supp f \subset V_1$ and $\supp g \subset V_2$, 
where 
\[
V_i = \{(\xi'',\sigma,\tau):  |\xi''|^2=2\sigma\tau+O(1),\, \tau\sim R,\,
|\xi''/\tau +(-1)^i 2^{-j}e_1|\le 2^{-j-2}\}, \,\, i=1,2. 
\]

Now, using  an anisotropic transformation $T_j : (\xi'',\sigma,\tau) \mapsto (2^{j}\xi'',\sigma,2^{2j}\tau)$,
we have
\begin{align*}
\|  \widehat {f}\widehat {g}\|_{L^{q/2}( d\mu)}=\| \widehat {f_j}\widehat
{g_j}\|_{L^{q/2}( d\mu_j)},
\end{align*}
where $f_j=|\det T_j| \,  f \circ T_j = 2^{(n+1)j} f \circ T_j  $,  $g_j=|\det T_j| \,  g\circ T_j $, 
and 
\[ \mu_j(\phi)=\int \phi\circ T_j \, d\mu.\]
So,  we see 
\[\supp f_j
\subset\Big\{(\xi'',\sigma,\tau):
|\xi''|^2=2\sigma\tau + O(2^{-2j}),\, \tau\sim 2^{-2j}R,\,
|\xi''/\tau - e_1|\le 2^{-2} \Big\}
\]
and 
\[\supp g_j
\subset\Big\{(\xi'',\sigma,\tau):
|\xi''|^2=2\sigma\tau + O(2^{-2j}),\, \tau\sim 2^{-2j}R,\,
|\xi''/\tau + e_1|\le 2^{-2} \Big\}.
\]

Clearly,  $\mu_j$ is supported in $\overline{B(0,1)}$ because $j\ge 1$, and  note that $f$ and $g$ are supported in $\Gamma_{2^{-2j}R}(2^{-2j})$ and satisfy the separation condition \eqref{angular-separation} in Theorem \ref{bifrac}, by which we have  that, for $ \beta > \beta(\alpha,q)$, 
\be\label{scaled}
\| \widehat {f_j}\widehat {g_j}\|_{L^{q/2}( d\mu_j)} \le C \langle \mu_j\rangle_\alpha^{\frac 2q} (2^{-2j}R)^{2\beta} \| f_j\|_2 \| g_j\|_2. 
\ee

We now estimate  $\langle \mu_j\rangle_{\alpha}$.  For any $( x,\rho) \in \mathbb R^{n+1}\times \mathbb R$, we have
\begin{align*} 
\mu_j( B(x,\rho))&=\int_{T_j^{-1} B(x,\rho)} d\mu
\lesssim \int_{\mathcal R} d\mu, 
\end{align*}
where $\mathcal R$ is a rectangle of dimensions $ 2^{-j}\rho \times2^{-j}\rho\times\cdots \times2^{-j}\rho\times \rho\times 2^{-2j}\rho$.
There are three different  ways of decomposing  $\mathcal R$ into cubes, namely, cubes  of side length $2^{-2j}\rho$, $2^{-j}\rho$, and $\rho$, respectively. Considering these three cases  and 
using  \eqref{def-measure} for each case,  we obtain 
\begin{align*}
 \mu_j(B(x,\rho))
&\lesssim \langle\mu\rangle_\alpha \,\min\{ 2^{(n+1-2\alpha)j}\rho^\alpha, \,2^{(1-\alpha)j}\rho^\alpha,\, \rho^\alpha \} \\
&= \langle\mu\rangle_\alpha \times \begin{cases}
\, 2^{(n+1-2\alpha)j}\rho^\alpha, &\text{ if }  n< \alpha \le n+1,\\
\, 2^{(1-\alpha)j}\rho^\alpha,    &\text{ if } 1<\alpha\le n,\\
\, \rho^\alpha,    &\text{ if }  0< \alpha \le 1.
\end{cases}
\end{align*}
Hence, $\langle \mu_j\rangle_\alpha \lesssim   \,\min\{ 2^{(n+1-2\alpha)j} , \,2^{(1-\alpha)j} ,\, 1 \} \times\langle\mu\rangle_\alpha$. 
Combining this with \eqref{scaled},  $\|f_j\|_2 = 2^{j(n+1)/2}\|f\|_2$ and $\|g_j\|_2 = 2^{j(n+1)/2}\|g\|_2$ yields  \eqref{sep}. 
\end{proof}

To handle the second sum in \eqref{decomp}, which is easier,  we use the following. 
\begin{lem}\label{diagonal}
	Suppose that $f$ is supported in $\Gamma_R(1)$  and the diameter of $\asupp f$ is $O(R^{-1/2})$. 
For $q \ge 2$, there exists a constant $C>0$ such that
\be\label{diagonaleq} 
\| \widehat {f}\|_{L^{q}(d\mu)}\le
C \langle \mu\rangle_\alpha^{\frac 1 q} \, R^{\beta_\circ(\alpha,q)}\, \|f\|_2,
\ee
where
\[\beta_\circ(\alpha,q) = \begin{cases}
(n+1)/4+(n+1-2\alpha)/2q, &\text{ if } n< \alpha \le n+1 ,\\
(n+1)/4+(1-\alpha)/2q,  &\text{ if } 1\le \alpha \le n,\\
(n+1)/4,  &\text{ if } 0<\alpha\le 1.
\end{cases}\]
\end{lem}	
\begin{proof}
Note that $f$ is supported in a rectangle $\mathcal R$ of dimensions 
$C \overset{n-1 \text{ times } }{R^{1/2} \times  \cdots \times CR^{1/2 } }\times C\times CR$ for a positive constant $C$. 
We consider $T_R$ which takes the unit cube $Q$ onto $\mathcal R$. 
Then we have
\[
\widehat f(x) = R^{\frac{n+1}{2}} \int e^{-i (T_R x) \cdot y} f_R (y) dy,
\]
where $f_R := f\circ T_R $ is supported in the unit cube $Q$.

As before, let $\mu_R$ be the measure given by
$
\mu_R (\phi) = \int \phi\circ T_R \, d\mu
$ for continuous function $\phi$.
Since $f_R$ is supported in $Q$, $|\widehat{f_R}| \le C(|\widehat{f_R}|^q\ast |\eta_{Q}|)^{\frac1q}$ for any Schwartz function $\eta_Q$ satisfying   $\eta_Q=1$ on $Q$. 
Using Hausdorff-Young inequality, Lemma \ref{fractal},  and H\"older inequality  we have, for $q \ge 2$, 
\begin{align*}
&\qquad\qquad \| \widehat f \,\|_{L^q(d\mu)} \lesssim R^{\frac{n+1}{2}} \| |\widehat{\eta_Q}|\ast d \mu_R\|_\infty^{\frac 1q} \, \| \widehat{f_R} \|_{L^q(\mathbb R^{n+1})}
\\&\lesssim R^{\frac{n+1}{2}} \langle \mu_R \rangle_\alpha^{\frac 1q} |\supp f_R|^{\frac12-\frac1q}  \,  \| {f_R} \|_{L^2(\mathbb R^{n+1})} 
\lesssim R^{\frac{n+1}{4}} \langle \mu_R \rangle_\alpha^{\frac 1q}  \, \| f \|_{L^2(\mathbb R^{n+1})}. 
\end{align*}
As in the proof of Theorem \ref{rescaling} (this  corresponds to the case $2^j\sim \sqrt R$), it is easy to see
\begin{align*}
 \langle \mu_R \rangle_\alpha^{\frac1q}
\lesssim \langle\mu\rangle_\alpha^{\frac1q} \times \begin{cases}
\, R^{\frac{n+1-2\alpha}{2q}}, &\text{ if }  n< \alpha \le n+1,\\
\, R^{\frac{1-\alpha}{2q}},    &\text{ if } 1<\alpha\le n,\\
\, 1,    &\text{ if }  0< \alpha \le 1.
\end{cases} 
\end{align*}
Therefore, combining these two estimates  gives \eqref{diagonaleq}.
\end{proof}


\begin{proof}[Proof of Theorem \ref{thecone1}]
By \eqref{decomp}, 
we have
\[
 \| \widehat f \|_{L^{q}(d\mu)}^2 \le \sum_{j=1}^{\log R^{1/2}}\sum_{(k,k^\prime):\theta^j_k\approx \theta^j_{k^\prime}}
\|\widehat {f_k^j}\widehat {f_{k'}^j}\|_{L^{q/2}(d\mu)} + \sum_{k} \|\widehat{f_k^{j_\circ}}\|_{L^{q}(d\mu)}^2 =: I + II.
\]

By Lemma \ref{rescaling}, for any $\epsilon >0$ and $R>1$, we have
\begin{align*}
I &\lesssim  \langle \mu \rangle_\alpha^{\frac 2q} R^{2\beta} \sum_{j=1}^{\log R^{1/2}} 2^{\gamma j} \sum_{(k,k^\prime):\theta^j_k\approx \theta^j_{k^\prime}} \| f_k^j\|_2 \| f_{k'}^j \|_2. 
\end{align*}
Hence,  by Schwarz inequality, we get 
\begin{align*}
 I \lesssim \langle \mu \rangle_\alpha^{\frac 2q} R^{2\beta} \times
\begin{cases}  \log R^{1/2}  \| f\|_2^2 \\
R^{\gamma /2} \log R^{1/2}  \| f\|_2^2 \end{cases}
\lesssim \langle \mu \rangle_\alpha^{\frac 2q} \times
\begin{cases}  R^{2\beta(\alpha,q) + \varepsilon}  \| f\|_2^2 , &\text{ if} \, \gamma \le 0, \\
R^{2 \beta_\circ(\alpha,q) + \varepsilon}  \| f\|_2^2 , &\text{ if} \, \gamma > 0. \end{cases}
\end{align*}
From Lemma \ref{diagonal}, we also see
\begin{align*}
II \lesssim \langle \mu\rangle_\alpha^{\frac2q} R^{2\beta_\circ(\alpha,q)} \sum_k \|f_k^{j_\circ}\|_2^2 \lesssim \langle \mu\rangle_\alpha^{\frac2q} R^{2\beta_\circ(\alpha,q)} \|f\|_2^2.
\end{align*}
Combining these estimates, we obtain that
for $ q \ge 2$, and for any $\epsilon>0$ and $R>1$, 
\begin{equation}\label{ndimension} 
\| \widehat f\|_{L^{q}(d\mu)}\lesssim
 \langle\mu\rangle_\alpha^{\frac 1q}\,R^{\widetilde s(\alpha,q,n) +\epsilon}\|f\|_2,
\end{equation}
where $\widetilde{s}(\alpha,q,n)=\max \{\beta(\alpha,q),\beta_\circ(\alpha,q)\}$. 
Note that
\[
\widetilde{s}(\alpha,q,n) =
\begin{cases}
\,    \max \{\frac {n}2-\frac\alpha q, \frac{n+1}4,\frac{3n+1}8-\frac\alpha 4\}, &\text{if } 0<\alpha\le1,\\
\,          \max\{\frac {n}2-\frac\alpha q, \frac{n+1}4+\frac{1-\alpha}{2q},\frac{3n+1}8-\frac\alpha 4\}, &\text{if } 1 <\alpha\le n,\\
\, \max \{ \frac {n}2-\frac\alpha q, \frac{n+1}4+\frac{n+1-2\alpha}{2q},\frac{3n+1}8-\frac\alpha 4 \},
             & \text{if } n <\alpha \le n+1.
           \end{cases}
\]

If $3\le n <\alpha$, we can improve $\widetilde{s}(\alpha,q,n)$ slightly. 
In fact, we use Plancherel theorem and Lemma \ref{fractal} so that, for $q \le 2$,  we have 
\begin{align*}
\|\widehat f\|_{L^q(d\mu)} \le \mu(\mathbb R^{n+1})^{\frac1q -\frac12} \|\widehat f\|_{L^2(d\mu)} \lesssim  \|\widehat f\|_{L^2} \| |\phi_R| \ast d\mu\|_\infty^{\frac12} \le  R^{\frac{n+1-\alpha}2}\|f\|_{L^2}. 
\end{align*}	
Since $\frac{ 3n+1}8-\frac\alpha 4 > \frac{n+1-\alpha}2$ if $ 3 \le n <\alpha \le n+1$, it follows that, for  $n <\alpha\le n+1$
\[
\widetilde{s}(\alpha,q,n) = \max \left\{ \frac {n}2-\frac\alpha q, \frac{n+1}4+\frac{n+1-2\alpha}{2q},\frac{n+1-\alpha}{2} \right\}.
\]
This completes the proof.
\end{proof}


\begin{proof}[Proof of Theorem \ref{wave}.] 
We now show that if \eqref{frac} holds for some $s=s_0$, then \eqref{fstrichartz} holds for $s>s_0$.  Since $\mu$ has compact support, by finite decomposition we may assume that $\mu$ is supported in $\overline{B(0,1)}$.   Recall that the solution $u(x,t)$ is given by
\begin{equation}\label{solution}
u(x,t) =  \frac{1}{(2\pi)^{n}}\int_{\mathbb R^n} e^{i x\cdot \xi} \cos (t |\xi|) \widehat f(\xi) d\xi
+  \frac{1}{(2\pi)^{n}} \int_{\mathbb R^n} e^{i x\cdot \xi} \sin (t |\xi|) \frac{\widehat g(\xi)}{|\xi|} d\xi.  
\end{equation}

Let $P_j$ be the projection operator $P_j$ for $j\ge1$ given by $
\widehat{ P_j f}(\xi) = \beta(| \xi | / 2^j) \widehat f(\xi),
$
where $\beta$ is a $C_0^\infty(\mathbb R)$ function which is supported in  $[1/2, 2]$ and $\sum_{j\in \mathbb Z}  \beta (|\xi|/2^j)=  1$ for $\xi\neq0$.
Also we define $P_{\le 0} f$ such that $\widehat{P_{\le 0}f}  = \beta_0 \widehat f$, where $\beta_0$ is a $C_0^{\infty}(\mathbb R)$ function such that $\beta_0(|\xi|) = 1- \sum_{j \ge 1} \beta(|\xi| / 2^j)$. 
We write 
\[u(x,t)=P_{\le 0}( u(\cdot,t))(x) +\sum_{j\ge 1} P_{j}( u(\cdot,t))(x).\] 
By Cauchy-Schwarz inequality and Plancherel theorem  we have
$|P_{\le 0}( u(\cdot,t))(x)| \lesssim \|f\|_2 +\|g\|_2.$
Since $\mu$ is supported in  $\overline{B(0,1)}$,  it follows that 
\[\|  P_{\le 0}( u(\cdot,t)) \|_{L^q(d\mu)} \lesssim  \|f\|_2 +\|g\|_2. \] 
So, in order to show  \eqref{fstrichartz} for $s>s_0$  it is sufficient to show that \eqref{frac} with $s=s_0$  implies 
\[\|  P_{j}( u(\cdot,t))  \|_{L^q(d\mu)} \lesssim    2^{s_0j}  \|f\|_2 + 2^{(s_0-1)j}\|g\|_2.\] 
By \eqref{solution} and time reversal symmetry this in turn follows from  
\begin{align}\label{decomp2}
\|  e^{i t \sqrt{-\Delta}} P_j  h  \|_{L^q(d\mu)} \le  2^{s_0j} \|h\|_2 .
\end{align}

Since $\mu$ is supported in $\overline{B(0,1)}$, as before, using the smooth function $\eta$ satisfying  $ \eta\sim 1$ on $\overline{B(0,1)}$ and $\supp \widehat\eta \subset
B(0,\frac12)$,  we have
  \begin{align*}
\|  e^{i t \sqrt{-\Delta}} P_j  h  \|_{L^q(d\mu)}\sim \|  \eta\, e^{i t \sqrt{-\Delta}} P_j  h  \|_{L^q(d\mu)}.
\end{align*}
Note that the space time Fourier transform of  $\eta\, e^{i t \sqrt{-\Delta}} P_j f (x)$ is supported in $\Gamma_{2^j}(1)$. Using \eqref{frac} with $s=s_0$  and Plancherel theorem  gives
  \begin{align*}
\|  e^{i t \sqrt{-\Delta}} P_j  h  \|_{L^q(d\mu)}\lesssim    2^{s_0j}  \|\eta\, e^{i t \sqrt{-\Delta}} P_j  h\|_{2} \lesssim   2^{s_0j} \|h\|_2 .
\end{align*}
Hence we get \eqref{decomp} and complete proof. 
\end{proof}

\section{Proof of Proposition \ref{prop:necessary}}\label{sec:sharpness}

Now, we obtain lower bounds on $s$ for which \eqref{frac} may hold. This is done by constructing suitable functions and measures.

\

Firstly we show that if \eqref{frac} holds, then 
\be \label{scaling}
s\ge \frac{n}2-\frac{\alpha}q. 
\ee 
Let $\mu$ be the measure given by
$d\mu(x)=\chi_{\overline{B(0,1)}}(x)|x|^{\alpha-n-1}dx$ and
$f=\chi_{\Gamma_R(1)}$. Then, $\mu$ is obviously $\alpha$-dimensional and $|\widehat{f}(x)|\gtrsim R^{n}$ if $|x|\le cR^{-1}$
with sufficiently small $c>0.$ Hence \eqref{frac} implies $
R^{n}R^{-\alpha/q}\le CR^sR^\frac {n}2$. So, letting $R \rightarrow \infty$ gives \eqref{scaling}. 

\

We  now show the second condition:
\begin{align} \label{second_1} s \ge
\begin{cases} 
\frac{n+1}{4},\, &\text{ if }\, 0 < \alpha \le 1,
\\
\frac{n+1}4 +\frac{1-\alpha}{2q},\, &\text{ if }\,1 < \alpha \le n, \\
\frac{n+1}4+ \frac{n+1-2\alpha}{2q},\, &\text{ if }\, n < \alpha \le n+1.
\end{cases}
\end{align}
Let us set 
\begin{equation}\label{P}
P=\Big\{(\xi_1,\xi'',\xi_{n+1}):
|\xi_1-\xi_{n+1}|\le \frac1{100},\, |{\xi''}|\le \frac{\sqrt{R}}{100},\, 
\frac54 R\le |\xi_1+\xi_{n+1}|\le \frac 32 R \Big \}.
\end{equation}
Here $\xi''\in \mathbb R^{n-1}$. Then $P$ is contained in $\Gamma_R(1)$. 
Let $x=(x_1, x'', x_{n+1})$ be the dual  variables of $(\eta_1,\xi'', 
\eta_{n+1})$, where $ \eta _1=(\xi_1+\xi_{n+1})/\sqrt 2$ and $\eta_{n+1}=(\xi_1-\xi_{n+1})/\sqrt 2$.

For a  given $\alpha$ let $\ell$ be a positive integer satisfying $\ell-1 <\alpha \le \ell$. 
We consider a measure $\mu$ defined by
\[
d \mu (x ) = \prod_{i=1}^{n+1-\ell} d\delta(x_i) |x_{n-\ell+2}|^{\alpha - \ell} d x_{n-\ell+2} d x_{n-\ell+3} \cdots d x_{n+1},
\]
where $\delta$ is the 1-dimensional delta measure and $x = (x_1,\dots,x_{n+1})\in\mathbb R^{n+1}$.
If $n <\alpha\le n+1$, we set
\[
d\mu(x) = |x_1|^{\alpha - n-1} dx_1 \cdots dx_{n+1}.
\]
Then it is easy to see that $\mu$ is $\alpha$-dimensional.  In fact, considering the delta measure,  $\mu(B(x_0,\rho)) \le C \rho^{\alpha -\ell +1} \cdot \rho^{\ell -1} = C \rho^\alpha$ for any $\rho >0$ and $x_0\in\mathbb R^{n+1}$.


Let $f = \chi_P$. Then we have $\| f \|_2 \lesssim R^{(n+1)/4}$. 
{We denote by} $P^*$ the {dual rectangle} of  $P$ of which dimensions are $C R^{-1} \times \overset{(n-1) \text{ times}}{C R^{-1/2}\times\cdots \times C R^{-1/2}} \times C $ for some constant $C$. 
{It follows that $|\widehat f(x) | \gtrsim R^{1+(n-1)/2}$ on a rectangle of which size is comparable to  $P^*$.}
{Hence,  \eqref{frac} gives 
\[
R^{1+\frac{n-1}2} \mu(P^*)^{\frac1q} \lesssim R^{s + \frac{n+1}4}.
\]
Letting $R \rightarrow \infty$, we obtain \eqref{second_1} because
\[
\mu(P^*) \approx 
\begin{cases}
\, 1,\,&\text{ if }\, 0<\alpha \le 1,\\
\, R^{-\frac 12(\alpha - \ell +1) }\times R^{-\frac12(\ell-2)},\,&\text{ if }\, 1< \alpha \le n,\\
\, R^{-(\alpha -n)}\times R^{-\frac{n-1}{2}},\,&\text{ if }\, n< \alpha\le n+1.
\end{cases}
\]}

\

{Finally we show that \eqref{frac} implies 
\begin{align} \label{third_1} s\ge
\begin{cases}
\,\frac{n+2}{4}-\frac{\alpha}{4},&\text{ if }\, 1< \alpha \le n,
\\
\, \frac{n+1}2-\frac{\alpha}{2},&\text{ if }\, n < \alpha \le n+1.
\end{cases}
\end{align}}

{The condition \eqref{third_1} can be obtained by} an adaptation of the example in \cite{er3} which was  based on
the one due to Wolff \cite{w2}. 

First we show \eqref{third_1} for $1 <\alpha \le n$. 
Let $\phi_P$ be a Schwartz function supported in $P$, where $P$ is given by \eqref{P}.
Let $N$ be an integer such that $ R^{\frac{\alpha-1}{2}}\sim  N$, and  let $v_1, \dots, v_N$ be the lattice points which are contained in $B_{n-1}(0,1)$ and separated by  distance $\sim R^{-\frac{\alpha-1}{2(n-1)}}$. 
Now we consider a Schwartz function $F$ supported in $\Gamma_R(1)$, which is given by
\[ F(\eta_1,\xi'', \eta_{n+1})= N^{-\frac12} \sum_{k=1}^N \phi_P(\xi)e^{iv_k\cdot\xi''}.\]
Here again, $\eta_1$ and $\eta_{n+1}$ are the coordinates defined as in the above. 
Note that $\widehat F$ is a sum of translations of $\widehat{\phi_P}$, i.e. $\widehat F = N^{-1/2} \sum_{k=1}^N \widehat{\phi_P}(x_1, x'' - v_k, x_{n+1})$. 
Since $1 < \alpha \le n$, we see $R^{-\frac{\alpha-1}{2(n-1)}}\ge R^{-\frac12}$, which implies that $P^* + v_k$'s are almost disjoint.
 By rapid decay of {$\widehat{\phi_P}$} outside of $P^*$ we see that $|\widehat
F|\gtrsim N^{-\frac12} R^\frac{n+1}2 \sim R^{\frac{2n+3-\alpha}4}$ on  $S
:= \bigcup_{k=1}^N(P^*+ v_k)$. Consider the measure $ d\mu
=R^{\frac{n+2-\alpha}2} \chi_{S}\,dx, $ which is an
$\alpha$-dimensional measure with $\langle\mu\rangle_\alpha\lesssim 1$ when
$1< \alpha\le n$.
In fact, for $R^{-1} \le \rho < R^{-1/2}$ and $x_\circ \in \mathbb R^{n+1}$ it is easy to see 
\begin{align*}
\mu(B(x_\circ,\rho))  \le R^{\frac{n+2-\alpha}2} |S\cap B(x_\circ,\rho)|
\le R^{\frac{n+2-\alpha}2} R^{-1} \rho^{n}
 \le \rho^{-(n-\alpha)}\rho^{n} = \rho^\alpha.
\end{align*}
The other cases $\rho<R^{-1}$, $R^{-1/2}<\rho \le 1$, and $\rho >1$ can be treated similarly.
Note that $ \int_S d\mu\sim 1$ and $\|F\|_2^2 \le N^{-1} \sum_{k=1}^N \|\phi_P\|_2^2 \le C R^{\frac{n+1}4}$.
Hence \eqref{frac} implies $R^{\frac{n+2-\alpha}2}\lesssim R^s$,
which gives \eqref{third_1}.


Now we proceed to show  \eqref{third_1} for $\alpha> n$.
Similarly as before, let $u_1, \dots, u_M$ be the lattice points
which are contained in $B_{n-1}(0,1)$ and separated by about
$R^{-\frac{2\alpha-n-1}{2(n+1)}}\gg R^{-\frac12}$ so that $M\sim
R^{\frac{(n-1)(2\alpha-n-1)}{2(n+1)}}$. Under the same assumption, let  $w_1, \dots, w_L$ be
the lattice points which are contained in $(-1/100,1/100)$ and
separated by $R^{-\frac{2\alpha-n-1}{n+1}}$ such that  $L\sim
R^{\frac{2\alpha-n-1}{n+1}} $. We set
\[ G(\eta_1, \xi'', \eta_{n+1})= (ML)^{-\frac12} \sum_{k=1}^M\sum_{j=1}^L
\phi_P(\xi)e^{i(u_k\cdot\xi''+w_j \eta_{n+1})}.\] 
Hence it follows that
$|\widehat G|\gtrsim (ML)^{-\frac12}
R^\frac{n+1}2=R^{\frac{3n+3-2\alpha}4}$ on  $T = \bigcup_{k=1}^M
\bigcup_{j=1}^L  (P^*+ u_k+ w_j)$. We now consider an $\alpha$-dimensional measure
$d\mu=R^{n+1-\alpha} \chi_{T}\, dx$. 
Noting $ \int_T d\mu\sim 1$ and $\|G\|_2\sim R^{\frac{n+1}4}$, we get the  second condition in   \eqref{third_1} by letting $R\to \infty$.

\bibliographystyle{plain}
\bibliography{NA_10966}

\end{document}